\documentclass[12pt]{amsart}
\usepackage{cases}
\usepackage{amscd}
\usepackage{amsmath}
\usepackage{amsthm}
\usepackage{amssymb}
\usepackage{mathrsfs}
\usepackage{amsfonts}
\usepackage{color}

\usepackage{pifont}

\usepackage{upgreek}
\usepackage{bm}

\usepackage{indentfirst, latexsym, bm, amsmath, eufrak, amsthm}%¶ÎÂäµÄÊ×ÐÐËõ½ø Òà¿ÉÔÚÿ¶ÎÇ°¼Ó\setlength{parindent}{Êý×Öem}
\usepackage{amsmath}
\usepackage{amsthm}
\usepackage{amssymb}
\usepackage{mathrsfs}
\usepackage{amsfonts}

\usepackage{pifont}

\usepackage{upgreek}
\usepackage{bm}
\newcommand*{\tr}{\mathrm{tr}}
\numberwithin{equation}{section}
\newtheorem{theo}{Theorem} %[section]

\newtheorem{lem}{Lemma}
\newtheorem{mcor}{Corollary}
\newtheorem{remark}{Remark}

\newtheorem{definition}{Definition}

\newcommand*{\D}[1]{\ensuremath{\nabla^{#1}}}

\begin{document}
\title{On isotropic Berwald scalar curvature}

%\author[Huitao Feng]{Huitao Feng$^1$}
\author[Ming Li]{Ming Li$^1$}

%\address{Huitao Feng:  Chern Institute of Mathematics \& LPMC, Nankai University, Tianjin
%300071, P. R. China}

%\email{fht@nankai.edu.cn}

\address{Ming Li: Mathematical Science Research Center,
Chongqing University of Technology,
Chongqing 400054, P. R. China }

\email{mingli@cqut.edu.cn}

%\thanks{$^1$~Partially supported by NSFC (Grant No. 11221091, 11271062, 11571184).}

\thanks{$^1$~Partially supported by NSFC (Grant No. 11871126) }

\maketitle

%\date{\today}
\begin{abstract}
In this short paper, we establish a closer relation between the Berwald scalar curvature and the $S$-curvature. In fact, we prove that a Finsler metric has isotropic Berwald scalar curvature  if and only if it has weakly isotropic $S$-curvature. For Finsler metrics of scalar flag curvature  and of weakly isotropic $S$-curvature, they have almost isotropic $S$-curvature if and only if the flag curvature is weakly isotropic.
\end{abstract}

%\tableofcontents

\section*{Introduction}

Let $(M,F)$ be an $n$-dimensional Finsler manifold.
%%$F$ is a non-negative continuous function defined on $TM$,  such that $F=F(x,y)$ is positively homogenous of degree one in $y$ variable, $C^{\infty}$ on $TM_0:=TM\setminus\{0\}$, and the Hessian form
%%\begin{equation*}
%%	g_{ij}(x,y)dy^i\otimes dy^j:=\tfrac{1}{2}[F^2]_{y^iy^j}dy^i\otimes dy^j
%%\end{equation*}
%%is positive definite on $TM_0$ along fibres.
The unit sphere bundle (or indicatrix bundle) is defined as $SM=\{(x,y)\in TM|F(x,y)=1\}$  with the natural projection $\pi:SM\to M$.
It is well known that the tangent bundle $T(SM)$ admits a natural horizontal subbundle $H(SM)$ and a Sasaki-type metric $g^{T(SM)}$. The Hilbert form $\omega^n=F_{y^i}dx^i$ defines a contact structure on $SM$. The Reeb field is just the spray $\mathbf{G}$.

Let $\tilde{P}$ be the Berwald curvature, which is the \textit{hv}-part of the curvature form of the Berwald connection $\nabla^{\rm B}$ on $H(SM)$.
Let $\mathbf{E}=\tr\tilde{P}$ be the mean Berwald curvature or the $\mathbf{E}$-curvature.  We introduce the Berwald scalar curvature as $\mathsf{e}:=\tr \mathbf{E}$ in \cite{Lizhang}, which is a function on $SM$ in general cases. If $\mathsf{e}=\pi^*c$ for some $c\in C^{\infty}(M)$, then
$F$ is said to have isotropic Berwald scalar curvature. The aim of this paper is to establish certain relations between the Berwald scalar curvature and the $S$-curvature of a Finsler manifold.
\begin{theo}\label{theo 1}
 Let $(M,F)$ be an $n$-dimensional Finsler manifold. Then $F$ has isotropic Berwald scalar curvature $\mathsf{e}$ if and only if $F$ has weakly isotropic $S$-curvature, and
 \begin{equation}\label{s wi}
 	S=\frac{1}{n-1}\mathsf{e}+\pi^*\xi(\mathbf{G}),
 \end{equation}
where $\xi\in\Omega^1(M)$. In both cases the mean Berwald curvature $\mathbf{E}$ is isotropic. Moreover, the following equation holds
  \begin{equation}
  \begin{split}
    \tr[\tilde{R}]=&\pi^*d\xi+\frac{1}{n-1}d\mathsf{e}\wedge\omega^n,
\end{split}
  \end{equation}
where $\tilde{R}$ is the \textit{hh}-part of the Berwald curvature form.
\end{theo}

A well known fact in Finsler geometry is that vanishing $S$-curvature implies that $\mathbf{E}=0$. Whether or not the converse is true is a natural question (c.f. \cite{SS}, p.67). We show that the converse is true under certain conditions.

\begin{theo}\label{theo 2}
	Let $(M,F)$ be an $n$-dimensional Finsler manifold. If the Berwald scalar curvature is constant and the mean stretch curvature $\bar{\Sigma}=0$, then the $S$-curvature of $F$ satisfies (\ref{s wi}) with $d\xi=0$, i.e., $S$ is almost constant.
\end{theo}
The following corollary can be obtained from Theorem \ref{theo 2} immediately.
\begin{mcor}\label{cor1}
  Let $M$ be a smooth manifold satisfies $H^1_{\rm dR}(M;\mathbf{R})=0$. If a Finsler structure $F$ on $M$ satisfies $\mathbf{J}=0$ and $\mathsf{e}=0$, then $S=0$.
\end{mcor}

For weakly Landsberg manifold, we have the following relations among the Berwald scalar curvature, $S$-curvature and the Cartan type 1-form.

\begin{mcor}\label{cor2}
		Let $(M,F)$ be an $n$-dimensional Finsler manifold with vanishing mean Landsberg curvature $\mathbf{J}=0$. Then the following statements are equivalent:
		\begin{enumerate}
			\item[(1)] $F$ has constant Berwald scalar curvature $\mathsf{e}$;
			\item[(2)] $F$ has almost constant $S$-curvature;
			\item[(3)] The Cartan type 1-form $d\eta=cd\omega^n$ for certain constant $c$.
		\end{enumerate}
\end{mcor}

To study the Finsler metrics of scalar flag curvature, i.e., $\mathbf{K}=\mathbf{K}(x,y)$,  is an important topic in Finsler geometry. Many authors contribute on this subject. One refers \cite{ChengShen,ChernShen,SS,Shenbook1,Shenbook2} and finds the references therein. Applying Theorem \ref{theo 1} to the Finsler metrics of scalar flag curvature and of weakly isotropic $S$-curvature, we obtain the following result.

\begin{mcor}\label{cor3}
  Let $(M,F)$ be a $n$-dimensional Finsler manifold of scalar flag curvature and weakly isotropic $S$-curvature. Then $F$ has almost isotropic $S$-curvature if and only if the flag curvature is weakly isotropic and satisfies
  \begin{equation}
    \mathbf{K}=\frac{3}{n^2-1}(\pi^*dc)(\mathbf{G})+\pi^*\sigma,
  \end{equation}
  where the $c$ and $\sigma$ are function on $M$, and $c$ is the Berwald scalar curvature $\mathsf{e}$ up to a constant.
  \end{mcor}

The necessary part of Corollary \ref{cor3} first appeared in \cite{ChMS} and plays an important role in the study of Finsler geometry.

\medskip

In this paper we adopt the index range and notations:
$$1\leq i,j,k,\ldots\leq n, \quad 1\leq \alpha,\beta,\gamma,\ldots\leq n-1, \quad 1\leq A,B,C,\ldots\leq 2n-1,$$

and $\bar{\alpha}:=n+\alpha,\bar{\beta}:=n+\beta,\bar{\gamma}:=n+\gamma,\ldots.$

\medskip

{\bf Acknowledgements.}  The author
would like to thanks Professor Huitao Feng for his consistent support and encouragement.

\section{Preliminary results}

We will introduce the necessary definitions and background of Finsler geometry in this section (c.f. \cite{BaoChernShen,ChernShen,Mo,SS,Shenbook1,Shenbook2} for more details).

Let $(M,F)$ be an $n$-dimensional Finsler manifold. Let $SM=\{F(x,y)=1\}$ be the unit sphere bundle (or indicatrix bundle) with the natural projection $\pi:SM\to M$.
Let $\{\mathbf{e}_1,\ldots,\mathbf{e}_n,\mathbf{e}_{n+1},\ldots,\mathbf{e}_{2n-1}\}$ be a local adapted orthonormal frame with respect to the Sasaki-type Riemannian metric
\begin{align*}
	g^{T(SM)}=\sum_{i}\omega^i\otimes\omega^i+\sum_{\bar{\alpha}}\omega^{\bar{\alpha}}\otimes\omega^{\bar{\alpha}}=:g+\dot{g}
\end{align*}
on $T(SM)$, where $\mathbf{e}_n=\mathbf{G}$ is the Reeb vector field. Let $\theta=\left\{\omega^1,\ldots,\omega^n,\omega^{n+1},\ldots,\omega^{2n-1}\right\}$ be the dual frame, then $\omega^n=F_{y^i}dx^i$ is the Hilbert form. $\omega^n$ defines a contact structure on $SM$. Let $\mathfrak{l}$ be the trivial bundle generated by $\mathbf{G}$. Let $\mathcal{D}$ be the contact distribution $\{\omega^n=0\}$. We define an almost complex structure on $\mathcal{D}$ by $$J=-\omega^\alpha\otimes\mathbf{e}_{\bar{\alpha}}+\omega^{\bar{\alpha}}\otimes\mathbf{e}_\alpha.$$
We will extend $J$ to be an endomorphism  $T(SM)$ by defining $J(\mathbf{e}_n)=0$. According to the definition of a contact metric structure in \cite{Blai}, one only need to verify that $(SM,J,\mathbf{G},\omega,g^{T(SM)})$ gives rise a contact metric structure.

Let $\mathscr{F}:=V(SM)$ be the integrable distribution given by the tangent spaces of the fibers of $SM$.  Let $p:T(SM)\to \mathscr{F} $ be the natural projection. Then the tangent bundle $T(SM)$ admits a splitting
\begin{equation}\label{splitting}
  T(SM)=\mathfrak{l}\oplus J\mathscr{F}\oplus \mathscr{F}=:H(SM)\oplus \mathscr{F}.
\end{equation}
The horizontal subbundle $H(SM)=\mathfrak{l}\oplus J\mathscr{F}$ is spanned by $\{\mathbf{e}_1,\ldots,\mathbf{e}_n\}$ on which the Chern connection is defined. And $\{\mathbf{e}_{n+1},\ldots,\mathbf{e}_{2n-1}\}$ gives a local frame of $\mathscr{F}$.

Let $\nabla^{\rm Ch}:\Gamma(H(SM))\to\Omega^1(SM;H(SM))$ be the Chern connection, which can be extended to a map
\begin{align*}
\nabla^{\rm Ch}:\Omega^*(SM;H(SM))\rightarrow \Omega^{*+1}(SM;H(SM)),
\end{align*}
where $\Omega^*(SM;H(SM)):=\Gamma(\Lambda^*(T^*SM)\otimes H(SM))$ denotes the horizontal valued differential forms on $SM$.
It is well known that the symmetrization of Chern connection  $\hat{\nabla}^{\rm Ch}$ is the Cartan connection.
The difference between $\hat{\nabla}^{\rm Ch}$ and $\nabla^{\rm Ch}$ will be referred as the Cartan endomorphism,
\begin{align*}
H= \hat{\nabla}^{\rm Ch}-\nabla^{\rm Ch}\quad\in\Omega^1(SM,{\rm
End}(H(SM))).
\end{align*}
Set $H=H_{ij}\omega^j\otimes\mathbf{e}_i$. By Lemma 3 and Lemma 4 in \cite{FL}, $H_{ij}=H_{ji}=H_{ij\gamma}\omega^{\bar{\gamma}}$ has the following form under natural coordinate systems
\begin{align*}
H_{ij\gamma}=-A_{pqk}u_i^pu_j^qu_{\gamma}^k,%  \label{H under natrual basis}
\end{align*}
where $A_{ijk}=\frac{1}{4}F[F^2]_{y^iy^jy^k}$ and $u_i^j$ are the transformation matrix from adapted orthonormal frames to natural frames.

Set $\eta={\rm tr}[H]~\in\Omega^1(SM).$
It is referred as the Cartan-type form in \cite{FL}. The Cartan-type form has the following local formula
\begin{align*}
\eta=\sum_{i=1}^nH_{ii\gamma}\omega^{\bar{\gamma}}=:H_{\gamma}\omega^{\bar{\gamma}}.
\end{align*}

Let $\bm{\omega}=(\omega_j^i)$ be the connection matrix of the Chern
connection with respect to the local adapted orthonormal frame field,
i.e.,
\begin{align*}
\D{\rm Ch}\mathbf{e}_i=\omega_i^j\otimes\mathbf{e}_j.%\label{400}
\end{align*}
\begin{lem}[\cite{BaoChernShen,ChernShen,Mo,SS}]\label{sturcture eq}
The connection matrix $\bm{\omega}=(\omega_j^i)$ of $\nabla^{\rm Ch}$ is determined by the following structure equations,
\begin{equation}\left\{
\begin{aligned}
&d\vartheta=-\bm{\omega}\wedge\vartheta,\\
&\bm{\omega}+\bm{\omega}^t=-2H,
\end{aligned}\right.\label{Chern connection structure eq. matrix}
\end{equation}
where $\vartheta=(\omega^1,\ldots,\omega^{n})^t$. Furthermore,
$$\omega_{\alpha}^{n}=-\omega^{\alpha}_{n}=\omega^{\bar{\alpha}},\quad{\rm and}\quad \omega^{n}_{n}=0.$$
\end{lem}

\begin{remark}
In \cite{FL}, we proved that the Chern connection is just the Bott connection on $H(SM)$ in the theory of foliation (c.f. \cite{Zhang}).
\end{remark}

Let $R^{\rm Ch}=\left(\nabla^{\rm Ch}\right)^2$  be the curvature of $\nabla^{\rm Ch}$.
Let $\Omega=\left(\Omega_j^i\right)$ be the curvature forms of $R^{\rm Ch}$. From the torsion freeness, the curvature form has no pure vertical differential form
\begin{align}\label{curvature chern connection}
\Omega_j^i:=d\omega_j^i-\omega_j^k\wedge\omega_k^i=\frac{1}{2}R_{j~kl}^{~i}\omega^k\wedge\omega^l+P_{j~k\gamma}^{~i}\omega^k\wedge\omega^{\bar{\gamma}}.
\end{align}
The Landsberg curvature is defined as
$$\mathbf{L}:=L^{~i}_{j~\gamma}\omega^{\bar{\gamma}}\otimes\omega^j\otimes\mathbf{e}_i=-P_{j~n\gamma}^{~i}\omega^{\bar{\gamma}}\otimes\omega^j\otimes\mathbf{e}_i,$$
the mean Landsberg curvature is defined by
$$\mathbf{J}={\rm tr}\mathbf{L}=J_{\gamma}\omega^{\bar{\gamma}}=-P_{i~n\gamma}^{~i}\omega^{\bar{\gamma}}.$$
The following formula is well known (c.f. \cite{BaoChernShen,ChernShen,Mo,SS})
\begin{align}\label{L=H|n}
	P_{n~k\gamma}^{~i}=-H_{ki\gamma|n}=-L_{ik\gamma}.
\end{align}
If a Finsler manifold satisfies $P=0$, $\mathbf{L}=0$ or $\mathbf{J}=0$, then it is called a Berwald, Landsberg or weak Landsberg manifold, respectively.

We will need parts of the Riemannian geometry of the sphere bundle, especially the geometry of the fibres. The following lemma can be found in \cite{Mo}.
\begin{lem}
Let $\Theta=(\Theta_B^A)$ be the connection form of the Levi-Civita connection $\nabla^{T(SM)}$ of $g^{T(SM)}$ with respect to the adapted frame $\{\mathbf{e}_A\}$. Then we have
	\begin{equation}
		\Theta_{\bar{\beta}}^{\bar{\gamma}}=\omega_{\beta}^\gamma+H_{\beta\gamma\alpha}\omega^{\bar{\alpha}},
	\end{equation}
where $\omega_{\beta}^\gamma$ are the connection forms of the Chern connection and $H_{\beta\gamma\alpha}\omega^{\bar{\alpha}}$ are the coefficients of the Cartan endomorphism.
\end{lem}

%%\begin{equation*}
%%\Theta=\left[ \begin{array}{cc}
%%\omega_j^i+\left(H_{ij\gamma}+\frac{1}{2}R_{n~ij}^{~\gamma}\right)\omega^{n+\gamma}&-\left(\frac{1}{2}R_{n~ij}^{~\alpha}+H_{ij\alpha}\right)\omega^{j}-P_{n~i\gamma}%%^{~\alpha}\omega^{n+\gamma}\\
%%\left(H_{ij\alpha}+\frac{1}{2}R_{n~ij}^{~\alpha}\right)\omega^{j}+P_{n~i\beta}^{~\alpha}\omega^{n+\beta}&\omega_{\beta}^\alpha+H_{\alpha\beta\gamma}\omega^{n+\gamma%%}
%%\end{array}\right].%\label{SM Levi-Civita connection form}
%%\end{equation*}

As the restriction the Levi-Civita connection $\nabla^{T(SM)}$ on $\mathscr{F}$, $\nabla^{\mathscr{F}}:=p\nabla^{T(SM)}p$ is the Euclidean connection of the bundle $(\mathscr{F},\dot{g})$. It is clear that along each fiber of $S_xM$, $x\in M$, $\nabla^{\mathscr{F}}$ is just the Levi-Civita connection of the Riemannian manifold $(S_xM, \dot{g}_x)$. We also define a connection $\bar{\nabla}^{\rm Ch}$ on $\mathscr{F}$ by
\begin{equation}\label{vertical ch connection}
  \bar{\nabla}^{\rm Ch}\mathbf{e}_{\bar{\beta}}=\omega_{\beta}^{\alpha}\otimes\mathbf{e}_{\bar{\alpha}}.
\end{equation}

We introduces some symbols to denote different covariant differentials for conveniens . For example, let $T=T^i_j\mathbf{e}_i\otimes\omega^j$ be an arbitrary smooth local section of the bundle $H(SM)\otimes H^*(SM)$ over $SM$. Then
\begin{equation*}
  \nabla^{\rm Ch}T=\left(dT^i_j+T^k_j\omega^i_k-T^i_k\omega^k_j\right)\otimes\mathbf{e}_i\otimes\omega^j%=:(\nabla^{\rm Ch}T)^i_j\otimes\mathbf{e}_i\otimes\omega^j.
\end{equation*}
If we expand the coefficients as one forms on $SM$ in terms of the adapted coframe, then we denote
\begin{equation*}
(\nabla^{\rm Ch}T)^i_j:=dT^i_j+T^k_j\omega^i_k-T^i_k\omega^k_j=:T^i_{j|\alpha}\omega^\alpha+T^i_{j|n}\omega^n+T^i_{j;\gamma}\omega^{\bar{\gamma}}.
\end{equation*}
Therefore we obtain
\begin{equation*}
 T^i_{j|\alpha}=i_{\mathbf{e}_{\alpha}}(\nabla^{\rm Ch}T)^i_j,\quad
 T^i_{j|n}=i_{\mathbf{e}_{n}}(\nabla^{\rm Ch}T)^i_j, \quad
 T^i_{j;\gamma}=i_{\mathbf{e}_{\bar{\gamma}}}(\nabla^{\rm Ch}T)^i_j,
\end{equation*}
where $i_{v}$ is the notation for the interior multiplication on differential forms by any vector $v$.

According to the splitting (\ref{splitting}) and using the almost complex structure $J$, we obtain a section of $\mathscr{F}\otimes\mathscr{F}^*$ from $T$ as following
\begin{equation*}
  \bar{T}=T^{\alpha}_{\beta}\mathbf{e}_{\bar{\alpha}}\otimes\omega^{\bar{\beta}}.
\end{equation*}
By (\ref{vertical ch connection}), the covariant differential of $\bar{T}$ by using $\bar{\nabla}^{\rm Ch}$ is given by
\begin{equation*}
  \bar{\nabla}^{\rm Ch}\bar{T}=\left(dT^{\alpha}_{\beta}+T^{\mu}_{\beta}\omega^{\alpha}_{\mu}-T^{\alpha}_{\mu}\omega^{\mu}_{\beta}\right)\otimes\mathbf{e}_{\bar{\alpha}}\otimes\omega^{\bar{\beta}}%=:(\nabla^{\rm Ch}T)^i_j\otimes\mathbf{e}_i\otimes\omega^j.
\end{equation*}
Similarly, we denote that
\begin{equation*}
  (\bar{\nabla}^{\rm Ch}\bar{T})^{\alpha}_{\beta}:=dT^{\alpha}_{\beta}+T^{\mu}_{\beta}\omega^{\alpha}_{\mu}-T^{\alpha}_{\mu}\omega^{\mu}_{\beta}=T^{\alpha}_{\beta\parallel\gamma}\omega^{\gamma}+T^{\alpha}_{\beta\parallel n}\omega^{n}+T^{\alpha}_{\beta,\gamma}\omega^{\bar{\gamma}}.
\end{equation*}
Therefore, by Lemma \ref{sturcture eq}, one has the following relations
\begin{equation}\label{covar ch bar ch}
  \begin{split}
    T^{\alpha}_{\beta\parallel i}&=i_{\mathbf{e}_{i}}(dT^{\alpha}_{\beta}+T^{\mu}_{\beta}\omega^{\alpha}_{\mu}-T^{\alpha}_{\mu}\omega^{\mu}_{\beta})=i_{\mathbf{e}_{i}}(dT^{\alpha}_{\beta}+T^{k}_{\beta}\omega^{\alpha}_{k}-T^{\alpha}_{k}\omega^{k}_{\beta})= T^{\alpha}_{\beta| i}\\
    T^{\alpha}_{\beta,\gamma}&=i_{\mathbf{e}_{\bar{\gamma}}}(dT^{\alpha}_{\beta}+T^{\mu}_{\beta}\omega^{\alpha}_{\mu}-T^{\alpha}_{\mu}\omega^{\mu}_{\beta})\\
    &=i_{\mathbf{e}_{\bar{\gamma}}}(dT^{\alpha}_{\beta}+T^{k}_{\beta}\omega^{\alpha}_{k}-T^{\alpha}_{k}\omega^{k}_{\beta})-i_{\mathbf{e}_{\bar{\gamma}}}T^n_{\beta}\omega_n^{\alpha}+i_{\mathbf{e}_{\bar{\gamma}}}T^{\alpha}_n\omega^n_{\beta}\\
    &= T^{\alpha}_{\beta;\gamma}+T^n_{\beta}\delta^\alpha_\gamma-T_n^\alpha\delta_{\beta\gamma}.
  \end{split}
\end{equation}
It is obvious that if $T$ is a section $J\mathscr{F}\otimes (J\mathscr{F})^*$, then $T^{\alpha}_{\beta;\gamma}=T^{\alpha}_{\beta,\gamma}$. These relations will be used in the following study.

The Berwald connection can be represented as $\nabla^{\rm B}=\nabla^{\rm Ch}+J^*\mathbf{L}$, where $J^*$ is dual of the almost complex structure $J$. Let $\tilde{\bm{\omega}}=(\tilde{\omega}_j^i)$ denote the Berwald connection form, then
\begin{align*}
  \tilde{\omega}_j^i=\omega_j^i-L^i_{j\alpha}\omega^\alpha.
\end{align*}

From the torsion freeness the Chern connection and (\ref{L=H|n}),
the Berwald connection is torsion free,
\begin{equation*}
  d\omega^i=\omega^j\wedge\tilde{\omega}_j^i.
\end{equation*}
Let  $\tilde{\Omega}_j^i$ be the curvature forms of the Berwald connection.
By the torsion freeness of the Berwald connection,  we have
\begin{align*}
\tilde{\Omega}_j^i:=d\tilde{\omega}_j^i-\tilde{\omega}_j^k\wedge\tilde{\omega}_k^i=:\frac{1}{2}\tilde{R}_{j~kl}^{~i}\omega^k\wedge\omega^l+\tilde{P}_{j~k\gamma}^{~i}\omega^k\wedge\omega^{\bar{\gamma}},
\end{align*}
where $\tilde{R}_{j~kl}^{~i}=-\tilde{R}_{j~lk}^{~i}.$ By using Lemma \ref{sturcture eq}, the curvature relations between the Chern connection and the Berwald connection are as following. The $hh$-curvatures satisfy
\begin{align}\label{hh curvature Berwald by Chern}
\begin{split}
  &\tilde{R}^{~\alpha}_{\beta~\gamma\mu}=R^{~\alpha}_{\beta~\gamma\mu}-(L^\nu_{\beta\gamma}L^\alpha_{\nu\mu}-L^\nu_{\beta\mu}L^\alpha_{\nu\gamma})+(L^\alpha_{\beta\gamma|\mu}-L^\alpha_{\beta\mu|\gamma}),\\
  &\tilde{R}^{~\alpha}_{\beta~\gamma n}= R^{~\alpha}_{\beta~\gamma n}+L^{\alpha}_{\beta\gamma|n},\quad \tilde{R}^{~\alpha}_{n~kl}=-\tilde{R}^{~n}_{\alpha~kl}=R^{~\alpha}_{n~kl}.
    \end{split}
\end{align}

The mean stretch cuvature
 \begin{equation}\label{msc}
	\bar{\Sigma}:=2\tr[\tilde{R}-R]
\end{equation}
is introduced and studied in \cite{NajafiTa}.

The relations between the $hv$-curvatures of the Berwald connection and the Chern connection are given by
\begin{align}\label{hv curvature Berwald by Chern}
	\begin{split}
		&\tilde{P}^{~\alpha}_{\beta~\gamma\mu}=P^{~\alpha}_{\beta~\gamma\mu}+L^\alpha_{\beta\gamma;\mu},\quad \tilde{P}^{~n}_{\alpha~\gamma\mu}=2L_{\alpha\gamma\mu},\\
		&\tilde{P}^{~\alpha}_{\beta~n\mu}=0,\quad \tilde{P}^{~\alpha}_{n~k\gamma}=0,\quad \tilde{P}^{~n}_{\alpha~n\gamma}=0.
	\end{split}
\end{align}

The mean Berwald curvature or the $\mathbf{E}$-curvature is defined by
\begin{equation*}
  \mathbf{E}:=\tr \tilde{P}=:E_{\gamma\mu}\omega^\gamma\wedge\omega^{\bar{\mu}}.
\end{equation*}
Under the local adapted frame, the coefficients of the $\mathbf{E}$-curvature is represented as
\begin{equation}\label{coefficients E}
  E_{\gamma\mu}=\tilde{P}^{~i}_{i~\gamma\mu}=\tilde{P}^{~\alpha}_{\alpha~\gamma\mu}=P^{~\alpha}_{\alpha~\gamma\mu}+L^\alpha_{\alpha\gamma;\mu}=P^{~\alpha}_{\alpha~\gamma\mu}+J_{\gamma;\mu}=P^{~\alpha}_{\alpha~\gamma\mu}+J_{\gamma,\mu}.
\end{equation}
We deliberately omit the factors $\frac{1}{2}$ and $F^{-1}$ in the usual definition of $\mathbf{E}$-curvature as a tensor on $TM_0$ in literatures.

\

\section{The $S$-curvature and the Berwald scalar curvature}

In this section we develop the relations between the $S$-curvature and the Berwald scalar curvature $\mathsf{e}$. We first state some necessary definitions.

Let $dV_M$ be any volume form of $M$. On a local coordinate chart $(U;x^i)$, $dV_M=\sigma(x)dx^1\wedge\cdots\wedge dx^n$.
The following important function on $SM$ is well defined,
$$\tau=\ln\frac{\sqrt{\det{(g_{ij})}}}{\sigma(x)}.$$
$\tau$ is called the distortion of $(M,F)$. The derivative of $\tau$ along the vector field $\mathbf{G}$ will be denoted by $S:=\mathbf{G}(\tau)$. The function $\mathbf{S}:=F\cdot S$ defined on $TM_0$ is called the $S$-curvature. In this paper, we will also call $S$ the $S$-curvature for convenient.
These two important concepts were first introduced by Zhongmin Shen in \cite{shen97}.
\begin{remark}
	In the study \cite{Li}, we show that the centro-affine differential geometric structure of the fibres of $SM$ naturally fit into the overall geometry of the fibre bundle $SM$. In affine geometry, the Tchebychev potential is the logarithm of the Radon-Nikodym derivative of the one induced by the affine metric and the measure induced by the embedding. One easily shows that the restriction of $-\tau$ on each fiber of $SM$ is just the Tchebychev potential of the fiber. Therefore the distortion can be naturally considered as the family version of the Tchebychev potential. One consults \cite{LSZ, NS, SSV} for more details about the affine differential geometry.
\end{remark}

In this paper, we will use the following definitions.
\begin{definition}[\cite{ChengShen,SS}]
	Let $(M,F)$ be an $n$-dimensional Finsler manifold.
	\begin{enumerate}
		\item[(i)] If the $S$-curvature satisfies
		\begin{equation}\label{WIS}
			S=\frac{1}{n-1}\pi^*c+\pi^*\xi(\mathbf{G})
		\end{equation}
		for some $c\in C^{\infty}(M)$ and $\xi\in\Omega^1(M)$, then $F$ is said to have weakly isotropic $S$-curvature;
		\item[(ii)] If $S$ satisfies (\ref{WIS}) and $d\xi=0$, then it is called to be almost isotropic;
		\item[(iii)] If $S$ satisfies (\ref{WIS}) and $\xi=0$, then it is called to be isotropic.
		%\item[(iv)] If $F$ has almost isotropic $S$-curvature with $c=0$, then then $F$ is said to have almost vanishing $S$-curvature;
		%\item[(v)] If $F$ has isotropic $S$-curvature with $c=0$, then then $F$ is said to have vanishing $S$-curvature.
	\end{enumerate}
\end{definition}

Let $\widetilde{dV_M}=e^{-f}dV_M$ be any other volume form on $M$, then the related distortion $\tilde{\tau}=\tau+\pi^*f$, where $f\in C^{\infty}(M)$. Thus the $S$-curvature related to $\widetilde{dV_M}$ is given by
\begin{equation}
	\tilde{S}=S+\pi^*df(\mathbf{G}).
\end{equation}
Therefore, the $S$-curvatures of a Finsler manifold are determined up to the exact 1-forms on the base manifold $M$.
If  $H^1_{\rm dR}(M;\mathbf{R})=0$, then the conditions almost isotropic $S$-curvature and isotropic $S$-curvature are equivalent.

Let $f\in C^{\infty}(SM)$ be any smooth function on $SM$, then $df$ admits the following decomposition with respect to (\ref{splitting}),
\begin{equation}
	df=f_{|\alpha}\omega^{\alpha}+f_{|n}\omega^{n}+f_{,\alpha}\omega^{\bar{\alpha}},
\end{equation}
where we denote $f_{|i}:=\mathbf{e}_i (f)$, and  $f_{,\alpha}:=\mathbf{e}_{\bar{\alpha}}(f)$.

\begin{lem}\label{d2f=0}
	From the fact $d^2f=0$, one has the following identities,
	\begin{align}
		&f_{|\alpha|\beta}-f_{|\beta|\alpha}+f_{,\gamma}R^{~\gamma}_{n~\alpha\beta}=0,\label{basic equ 1 d2f}\\
		&f_{|\alpha|n}-f_{|n|\alpha}+f_{,\gamma}R^{~\gamma}_{n~\alpha n}=0,\label{basic equ 2 d2f}\\
		&f_{|\alpha,\beta}-f_{|n}\delta_{\alpha\beta}+f_{,\gamma}P^{~\gamma}_{n~\alpha\beta}-f_{,\beta|\alpha}=0,\label{basic equ 3 d2f}\\
		&f_{|\alpha}+f_{|n,\alpha}-f_{,\alpha|n}=0,\label{basic equ 4 d2f}\\
		&f_{,\alpha,\beta}-f_{,\beta,\alpha}=0. \label{basic equ 5 d2f}
	\end{align}
The following useful equations can be derived from the above identities
\begin{align}
		&(f_{|n})_{,\alpha|n}+(f_{|n})_{|\alpha}=f_{,\alpha|n|n}+f_{,\gamma}R^{~\gamma}_{n~\alpha n},\label{basic equ 6 d2f}\\
		&(f_{|n})_{,\alpha,\beta}+(f_{|n})\delta_{\alpha\beta}=(f_{,\gamma}P^{~\gamma}_{n~\alpha\beta}-f_{,\beta|\alpha})+f_{,\alpha|n,\beta}.\label{basic equ 7 d2f}
\end{align}
%If $f$ is a function on $M$, then $df=f_{\alpha}\omega^\alpha+f_{n}\omega^n$ and the
\end{lem}

\begin{proof}
	By using the properties of the Chern connection and the notations in Section 1, we have
	\begin{equation}\label{dd^JF f}
		\begin{split}
				&d(f_{|\alpha}\omega^\alpha)=(d(f_{|\alpha})-f_{|\beta}\omega^\beta_\alpha)\wedge\omega^\alpha-f_{|\beta}\omega^\beta_n\wedge\omega^n\\
			=&f_{|\alpha|\beta}\omega^\beta\wedge\omega^\alpha+f_{|\alpha|n}\omega^n\wedge\omega^\alpha+f_{|\alpha,\beta}\omega^{\bar{\beta}}\wedge\omega^\alpha+f_{|\alpha}\omega^{\bar{\alpha}}\wedge\omega^n\\
			=&-\frac{1}{2}(f_{|\alpha|\beta}-f_{|\beta|\alpha})\omega^\alpha\wedge\omega^\beta-f_{|\alpha|n}\omega^\alpha\wedge\omega^n-f_{|\alpha,\beta}\omega^\alpha\wedge\omega^{\bar{\beta}}-f_{|\alpha}\omega^n\wedge\omega^{\bar{\alpha}},
		\end{split}
	\end{equation}
		\begin{equation}\label{dd^l f}
		\begin{split}
			&d(f_{|n}\omega^n)=f_{|n|\alpha}\omega^\alpha\wedge\omega^n+f_{|n,\alpha}\omega^{\bar{\alpha}}\wedge\omega^n+f_{|n}\delta_{\alpha\beta}\omega^\alpha\wedge\omega^{\bar{\beta}}\\
			=&f_{|n|\alpha}\omega^\alpha\wedge\omega^n+f_{|n}\delta_{\alpha\beta}\omega^\alpha\wedge\omega^{\bar{\beta}}-f_{|n,\alpha}\omega^n\wedge\omega^{\bar{\alpha}},
		\end{split}
	\end{equation}
and
\begin{equation}\label{dd^F f}
	\begin{split}
		&d(f_{,\alpha}\omega^{\bar{\alpha}})=d(f_{,\alpha})\wedge\omega^{\bar{\alpha}}+f_{,\alpha}(-\Omega_n^\alpha-\omega_n^\beta\wedge\omega_\beta^\alpha)\\
		=&(d(f_{,\alpha})-f_{,\beta}\omega_\alpha^\beta)\wedge\omega^{\bar{\alpha}}-f_{,\alpha}\Omega_n^\alpha\\
		=&f_{,\alpha|\beta}\omega^\beta\wedge\omega^{\bar{\alpha}}+f_{,\alpha|n}\omega^n\wedge\omega^{\bar{\alpha}}+f_{,\alpha,\beta}\omega^{\bar{\beta}}\wedge\omega^{\bar{\alpha}}\\
		&-\frac{1}{2}f_{,\alpha}R^{~\alpha}_{n~\beta\gamma}\omega^\beta\wedge\omega^\gamma-f_{,\alpha}R^{~\alpha}_{n~\beta n}\omega^\beta\wedge\omega^n-f_{,\alpha}P^{~\alpha}_{n~\beta\gamma}\omega^\beta\wedge\omega^{\bar{\gamma}}\\
		=&-\frac{1}{2}f_{,\alpha}R^{~\alpha}_{n~\beta\gamma}\omega^\beta\wedge\omega^\gamma-f_{,\alpha}R^{~\alpha}_{n~\beta n}\omega^\beta\wedge\omega^n+(f_{,\gamma|\beta}-f_{,\alpha}P^{~\alpha}_{n~\beta\gamma})\omega^\beta\wedge\omega^{\bar{\gamma}}\\
		&+f_{,\alpha|n}\omega^n\wedge\omega^{\bar{\alpha}}-\frac{1}{2}(f_{,\alpha,\beta}-f_{,\beta,\alpha})\omega^{\bar{\alpha}}\wedge\omega^{\bar{\beta}}.
	\end{split}
\end{equation}
By taking the sum of (\ref{dd^JF f}), (\ref{dd^l f}) and (\ref{dd^F f}), we obtain
\begin{equation}\label{d2f}
	\begin{split}
		0=&-\frac{1}{2}(f_{|\alpha|\beta}-f_{|\beta|\alpha})\omega^\alpha\wedge\omega^\beta-f_{|\alpha|n}\omega^\alpha\wedge\omega^n-f_{|\alpha,\beta}\omega^\alpha\wedge\omega^{\bar{\beta}}-f_{|\alpha}\omega^n\wedge\omega^{\bar{\alpha}}\\
		&+f_{|n|\alpha}\omega^\alpha\wedge\omega^n+f_{|n}\delta_{\alpha\beta}\omega^\alpha\wedge\omega^{\bar{\beta}}-f_{|n,\alpha}\omega^n\wedge\omega^{\bar{\alpha}}\\
		&-\frac{1}{2}f_{,\alpha}R^{~\alpha}_{n~\beta\gamma}\omega^\beta\wedge\omega^\gamma-f_{,\alpha}R^{~\alpha}_{n~\beta n}\omega^\beta\wedge\omega^n+(f_{,\gamma|\beta}-f_{,\alpha}P^{~\alpha}_{n~\beta\gamma})\omega^\beta\wedge\omega^{\bar{\gamma}}\\
		&+f_{,\alpha|n}\omega^n\wedge\omega^{\bar{\alpha}}-\frac{1}{2}(f_{,\alpha,\beta}-f_{,\beta,\alpha})\omega^{\bar{\alpha}}\wedge\omega^{\bar{\beta}}\\
		=&-\frac{1}{2}(f_{|\alpha|\beta}-f_{|\beta|\alpha}+f_{,\gamma}R^{~\gamma}_{n~\alpha\beta})\omega^\alpha\wedge\omega^\beta-(f_{|\alpha|n}-f_{|n|\alpha}+f_{,\gamma}R^{~\gamma}_{n~\alpha n})\omega^\alpha\wedge\omega^n\\
		&-(f_{|\alpha,\beta}-f_{|n}\delta_{\alpha\beta}+f_{,\gamma}P^{~\gamma}_{n~\alpha\beta}-f_{,\beta|\alpha})\omega^\alpha\wedge\omega^{\bar{\beta}}-(f_{|\alpha}+f_{|n,\alpha}-f_{,\alpha|n})\omega^n\wedge\omega^{\bar{\alpha}}\\
		&-\frac{1}{2}(f_{,\alpha,\beta}-f_{,\beta,\alpha})\omega^{\bar{\alpha}}\wedge\omega^{\bar{\beta}}.
	\end{split}
\end{equation}
Therefore, the coefficients in  (\ref{d2f}) with respect to the local adapted bases of $\Omega^2(SM)$ vanish. The identities (\ref{basic equ 1 d2f})-(\ref{basic equ 5 d2f}) are proved in this way.

Taking covariant derivative of (\ref{basic equ 4 d2f}) along $\mathbf{e}_n$ and substituting in (\ref{basic equ 2 d2f}) yields (\ref{basic equ 6 d2f}).  Taking covariant derivative of (\ref{basic equ 4 d2f}) along $\mathbf{e}_{\bar{\alpha}}$ and substituting in (\ref{basic equ 3 d2f}) yields (\ref{basic equ 7 d2f}).
\end{proof}

Applying Lemma \ref{d2f=0} to the distortion $\tau$ yields the following lemma.
\begin{lem}
The differential of $\tau$ is given by
\begin{equation}
d\tau=\tau_{|i}\omega^i+\eta=\tau_{|\alpha}\omega^\alpha+S\omega^n+\eta.\label{dtau}
\end{equation}
The following identities are valid
	\begin{align}
	&\tau_{|\alpha|\beta}-\tau_{|\beta|\alpha}+R_{i~\alpha\beta}^{~i}=0,\label{basic equ 1}\\
	&S_{|\alpha}-\tau_{|\alpha|n}-R_{i~\alpha n}^{~i}=0,\label{basic equ 2}\\
	&\tau_{|\gamma,\mu}-S\delta_{\gamma\mu}+P_{i~\gamma\mu}^{~i}=0,\label{basic equ 3}\\
	&S_{,\mu}+\tau_{|\mu}+P_{i~n\mu}^{~i}=0.\label{basic equ 4}
\end{align}
Furthermore, one has
\begin{equation}\label{Jacobi I}
  S_{|\alpha}-S_{,\alpha|n}=J_{\alpha|n}+R^{~i}_{i~\alpha n},
\end{equation}
and
\begin{equation}\label{hessian S}
	S_{,\gamma,\mu}+S\delta_{\gamma\mu}=E_{\gamma\mu}.
\end{equation}
\end{lem}
\begin{proof}
	These identities (\ref{basic equ 1})-(\ref{hessian S}) are direct consequences of Lemma \ref{d2f=0} by using (\ref{dtau}),  (\ref{coefficients E}) and some well known Bianchi identities or the following fact
	\begin{equation}\label{deta}
		d\eta=d(H_{\gamma}\omega^{\bar{\gamma}})=-\frac{1}{2}R_{i~kl}^{~i}\omega^k\wedge\omega^l-P_{i~k\gamma}^{~i}\omega^k\wedge\omega^{\bar{\gamma}}.
	\end{equation}
\end{proof}

Let $\xi\in\Omega^1(M)$ be a 1-form. In a local coordinate system, $\xi=\xi_i(x)dx^i$. Let
\begin{equation}\label{dx=omega}
	\pi^*dx^i=u^i_j\omega^j=u^i_\alpha\omega^\alpha+\frac{y^i}{F}\omega^n
\end{equation}
be the relationship between the natural bases and the orthonormal ones. By using (\ref{dx=omega}), $\pi^*\xi$ can be written as
\begin{equation}\label{pixi}
	\pi^*\xi=\xi_iu^i_\alpha\omega^\alpha+\frac{\xi_iy^i}{F}\omega^n=:\xi_\alpha\omega^\alpha+\xi_n\omega^n.
\end{equation}

The following lemma gives a representation of $\pi^*d\xi$ on $SM$.
\begin{lem}
	For any $\xi\in\Omega^1(M)$, one can express $\pi^*d\xi$ in the following form
	\begin{equation}\label{dxi}
		d\pi^*\xi=-\frac{1}{2}(\xi_{\alpha|\beta}-\xi_{\beta|\alpha})\omega^\alpha\wedge\omega^\beta-(\xi_{\alpha|n}-\xi_{n|\alpha})\omega^\alpha\wedge\omega^n.
	\end{equation}
Moreover, one has
\begin{equation}\label{aha}
	\xi_{\alpha,\beta}-\xi_{n}\delta_{\alpha\beta}=0.
\end{equation}
\end{lem}
\begin{proof}
	For any sections $\xi_{\alpha}\omega^\alpha\in\Gamma(J\mathscr{F}^*)$ and $\xi_{n}\omega^n\in\Gamma(\mathfrak{l}^*)$, one has
		\begin{equation}\label{d^JF xi}
		\begin{split}
			&d(\xi_{\alpha}\omega^\alpha)=(d(\xi_{\alpha})-\xi_{\beta}\omega^\beta_\alpha)\wedge\omega^\alpha-\xi_{\beta}\omega^\beta_n\wedge\omega^n\\
			=&\xi_{\alpha|\beta}\omega^\beta\wedge\omega^\alpha+\xi_{\alpha|n}\omega^n\wedge\omega^\alpha+\xi_{\alpha,\beta}\omega^{\bar{\beta}}\wedge\omega^\alpha+\xi_{\alpha}\omega^{\bar{\alpha}}\wedge\omega^n\\
			=&-\frac{1}{2}(\xi_{\alpha|\beta}-\xi_{\beta|\alpha})\omega^\alpha\wedge\omega^\beta-\xi_{\alpha|n}\omega^\alpha\wedge\omega^n-\xi_{\alpha,\beta}\omega^\alpha\wedge\omega^{\bar{\beta}}-\xi_{\alpha}\omega^n\wedge\omega^{\bar{\alpha}},
		\end{split}
	\end{equation}
and
	\begin{equation}\label{d^l xi}
		\begin{split}
			&d(\xi_{n}\omega^n)=\xi_{n|\alpha}\omega^\alpha\wedge\omega^n+\xi_{n,\alpha}\omega^{\bar{\alpha}}\wedge\omega^n+\xi_{n}\delta_{\alpha\beta}\omega^\alpha\wedge\omega^{\bar{\beta}}\\
			=&\xi_{n|\alpha}\omega^\alpha\wedge\omega^n+\xi_{n}\delta_{\alpha\beta}\omega^\alpha\wedge\omega^{\bar{\beta}}-\xi_{n,\alpha}\omega^n\wedge\omega^{\bar{\alpha}}.
		\end{split}
	\end{equation}
When $\xi\in\Omega^1(M)$, the derivative of $\xi_n$ along the direction $\mathbf{e}_{\bar{\alpha}}=-u^i_\alpha F\frac{\partial}{\partial y^i}$ is given by
\begin{equation}\label{dxin-xialpha}
	\xi_{n,\alpha}=\mathbf{e}_{\bar{\alpha}}(\xi_n)=-u^j_\alpha F\frac{\partial}{\partial y^j}(\frac{\xi_iy^i}{F})=-u^j_\alpha (\delta_j^i-\frac{F_{y^j}y^i}{F})\xi_i=-\xi_iu^i_\alpha=-\xi_\alpha.
\end{equation}
Thus the facts (\ref{d^JF xi})-(\ref{dxin-xialpha}) yield
\begin{equation}\label{dxi'}
	d\pi^*\xi=-\frac{1}{2}(\xi_{\alpha|\beta}-\xi_{\beta|\alpha})\omega^\alpha\wedge\omega^\beta-(\xi_{\alpha|n}-\xi_{n|\alpha})\omega^\alpha\wedge\omega^n-(\xi_{\alpha,\beta}-\xi_{n}\delta_{\alpha\beta})\omega^\alpha\wedge\omega^{\bar{\beta}}.
\end{equation}
Since $d\pi^*\xi=\pi^*d\xi$ does not involve vertical differential forms, we obtain the desired formulae from  (\ref{dxi'}).
\end{proof}

Now we present a proof of the theorems and corollaries.
\begin{proof}[Proof of Theorem \ref{theo 1}]
  The necessary part of the theorem has been proved in \cite{Lizhang}. The key ingredient is that along each fibre of $SM$, the elliptic equation
  \begin{equation*}
  	\Delta_{\dot{g}}f+\dot{g}(\eta,d^{SM/M}f)+(n-1)f=0,
  \end{equation*}
has only linear solutions (c.f.\cite{Schneider1,NS}). By using this fact and (\ref{hessian S}), in \cite{Lizhang}, we prove that if $\mathsf{e}$ is constant along each fiber of $SM$, then
$S$ is weakly isotropic and
 \begin{equation}
	S=\frac{1}{n-1}\mathsf{e}+\xi_i\frac{y^i}{F},
\end{equation}
where $\xi=\xi_i(x)dx^i\in\Omega^1(M)$.

Now we prove the sufficient part. Assume that the $S$-curvature is weakly isotropic and satisfies (\ref{WIS}) for some $c\in C^{\infty}(M)$ and $\xi\in\Omega^1(M)$.
The condition (\ref{WIS}) and the representation (\ref{pixi}) of $\xi$ shows that
\begin{equation}\label{xin WIS}
	\xi_n=S-\frac{1}{n-1}\pi^*c.
\end{equation}
By taking derivative of (\ref{xin WIS}) along the fibres, the formula (\ref{dxin-xialpha}) implies that
\begin{equation}\label{xialpha WIS}
	\xi_{\alpha}=-S_{,\alpha}.
\end{equation}
Substituting (\ref{xialpha WIS}) and (\ref{xin WIS}) in the formula (\ref{aha}) yields
\begin{equation}\label{aha WIS}
	S_{,\alpha,\beta}+(S-\frac{1}{n-1}\pi^*c)\delta_{\alpha\beta}=0.
\end{equation}
By (\ref{hessian S}) and (\ref{aha WIS}), one obtains that the $\mathbf{E}$-curvature is isotropic, i.e.,
\begin{equation}\label{E iso}
	E_{\alpha\beta}=\frac{1}{n-1}\pi^*c\delta_{\alpha\beta}.
\end{equation}
The equation (\ref{E iso}) obviously implies  the Berwald scalar curvature
\begin{equation}\label{e=pic}
  \mathsf{e}=\pi^*c.
\end{equation}

Now applying (\ref{xialpha WIS}),(\ref{xin WIS}) and (\ref{e=pic}) to the equation (\ref{dxi}) yields
\begin{equation}\label{pidxi WIS}
\pi^*d\xi=\frac{1}{2}(S_{,\alpha|\beta}-S_{,\beta|\alpha})\omega^\alpha\wedge\omega^\beta+(S_{,\alpha|n}+(S-\frac{1}{n-1}\mathsf{e})_{|\alpha})\omega^\alpha\wedge\omega^n.
\end{equation}

Applying (\ref{basic equ 1}), (\ref{basic equ 2}), (\ref{basic equ 4}) and (\ref{Jacobi I}) to (\ref{pidxi WIS}), we obtain
\begin{equation}\label{pidxi2 WIS}
	\begin{split}
\pi^*d\xi=&-\frac{1}{2}(\tau_{|\alpha|\beta}-\tau_{|\beta|\alpha})\omega^\alpha\wedge\omega^\beta+\frac{1}{2}(J_{\alpha|\beta}-J_{\beta|\alpha})\omega^\alpha\wedge\omega^\beta\\
	& +(S_{|\alpha}-\tau_{|\alpha|n}+J_{\alpha|n})\omega^\alpha\wedge\omega^n-\frac{1}{n-1}d\mathsf{e}\wedge\omega^n\\ =&\frac{1}{2}R^{~i}_{i~\alpha\beta}\omega^\alpha\wedge\omega^\beta+\frac{1}{2}(J_{\alpha|\beta}-J_{\beta|\alpha})\omega^\alpha\wedge\omega^\beta\\
	& +(R^{~i}_{i~\alpha n}+J_{\alpha|n})\omega^\alpha\wedge\omega^n-\frac{1}{n-1}d\mathsf{e}\wedge\omega^n.
	\end{split}
\end{equation}
The relations (\ref{hh curvature Berwald by Chern}) between the \textit{hh}-curvature forms of the Berwald connection and the Chern connection  and  (\ref{pidxi2 WIS}) yields
\begin{equation}\label{pidxi=r+s+de}
	\begin{split}
		\pi^*d\xi=\tr[\tilde{R}]-\frac{1}{n-1}d\mathsf{e}\wedge\omega^n
		=\tr[R]+\frac{1}{2}\bar{\Sigma}-\frac{1}{n-1}d\mathsf{e}\wedge\omega^n,
	\end{split}	
\end{equation}
where $\bar{\Sigma}$ is the mean stretch curvature (\ref{msc}).
\end{proof}

\begin{proof}[The proof of Theorem 2]
	Assume that the Finsler metric $F$ satisfies $d\mathsf{e}=0$ and $\bar{\Sigma}=0$. From (\ref{pidxi=r+s+de}) we obtain
	\begin{equation}\label{pidxi=r}
		\begin{split}
			\pi^*d\xi=\tr[R]=\frac{1}{2}H_{\gamma}R^{~\gamma}_{n~ab}\omega^a\wedge\omega^b.
		\end{split}
	\end{equation}
Since that $\eta=d^{SM/M}\tau$ and the fibres of $SM$ are compact, $\eta$ has zero points on each fibre of $SM$. However $\pi^*d\xi$ is constant along each fibre. We conclude that the both sides of the equation (\ref{pidxi=r}) vanish. Therefore $d\xi=0$  as $\pi$ is surjective. And $\tr[R]=0$.
\end{proof}

%The proof of Corollary \ref{cor1} is similar.
\begin{proof}[The proof of Corollary \ref{cor1}]
	We are going to prove the chain of statements: (1)$\implies$(2)$\implies$(3)$\implies$(1).
	
	(1)$\implies$(2): By the assumption $d\mathsf{e}=0$ and $\mathbf{J}=0$, a similar argument as the proof of Theorem \ref{theo 2} shows that the $S$-curvature is almost constant
	\begin{equation}\label{S ac}
		S=\frac{1}{n-1}\mathsf{e}+\pi^*\xi(\mathbf{G}),
	\end{equation}
with $d\xi=0$.

    (2)$\implies$(3): Assume that $S$ is given by (\ref{S ac}) with $d\mathsf{e}=d\xi=0$. The hypothesis $\mathbf{J}=0$ and (\ref{pidxi=r+s+de}) yields $\tr[R]=0$. By Theorem \ref{theo 1}, the hypothesis $\mathbf{J}=0$ and (\ref{deta}), we obtain $d\eta=-\tr[P]=-\mathbf{E}=-\mathsf{e}d\omega^n$.
      	
   (3)$\implies$(1): By (\ref{deta}), the proof is clear.
\end{proof}

\begin{proof}[The proof of Corollary \ref{cor3}]
	Assume that $F$ is of scalar curvature, the following formula (c.f. \cite{Mo}) is well known
	\begin{equation}\label{Jacobi 2}
		R^{~i}_{i~\alpha n}+J_{\alpha|n}=-\frac{n+1}{3}\mathbf{K}_{,\alpha}.
	\end{equation}

	Assume that $F$ has almost isotropic $S$-curvature. By (\ref{Jacobi 2}) and (\ref{pidxi2 WIS}), one obtains
	\begin{equation}\label{cor31}
		0=\frac{n+1}{3}\mathbf{K}_{,\alpha}\omega^\alpha\wedge\omega^n+\frac{1}{n-1}d\mathsf{e}\wedge\omega^n.
	\end{equation}
Since $\mathsf{e}$ is a function on $M$, (\ref{cor31}) and (\ref{basic equ 4 d2f}) yields
\begin{equation}
0=(\frac{n+1}{3}\mathbf{K}-\frac{1}{n-1}\mathsf{e}_{|n})_{,\alpha}\omega^\alpha\wedge\omega^n.
\end{equation}
Therefore $\frac{n+1}{3}\mathbf{K}-\frac{1}{n-1}d\mathsf{e}(\mathbf{G})=\pi^*\sigma$, for some $\sigma\in C^{\infty}(M)$.

Conversely,  $F$ has weakly isotropic flag curvature if and only if
\begin{equation}\label{cor32}
	-\frac{n+1}{3}\mathbf{K}_{,\alpha}=-\frac{1}{n-1}c_{|n,\alpha}=\frac{1}{n-1}c_{|\alpha},
\end{equation}
where $c\in C^{\infty}(M)$. Assume that $F$ has weakly isotropic $S$-curvature, (\ref{pidxi2 WIS}) and (\ref{cor32}) implies that
\begin{equation}\label{cor33}
	\begin{split} \pi^*d\xi =&\frac{1}{2}(R^{~i}_{i~\alpha\beta}+(J_{\alpha|\beta}-J_{\beta|\alpha}))\omega^\alpha\wedge\omega^\beta
		 +\frac{1}{n-1}d(c-\mathsf{e})\wedge\omega^n\\
	=&\frac{1}{2}\tilde{R}^{~i}_{i~\alpha\beta}\omega^\alpha\wedge\omega^\beta
	+\frac{1}{n-1}d(c-\mathsf{e})\wedge\omega^n.
	\end{split}
\end{equation}
Taking derivative of (\ref{cor33}), one obtains
\begin{equation*}
	\begin{split}
		0=&\frac{1}{2}(d\tilde{R}^{~i}_{i~\alpha\beta}\wedge\omega^\alpha\wedge\omega^\beta+\tilde{R}^{~i}_{i~\alpha\beta}d\omega^\alpha\wedge\omega^\beta-\tilde{R}^{~i}_{i~\alpha\beta}\omega^\alpha\wedge d\omega^\beta)\\
		&+\frac{1}{n-1}d(c-\mathsf{e})\wedge \sum_{\alpha}(\omega^\alpha\wedge\omega^{\bar{\alpha}})\\
		=&\frac{1}{2}(d\tilde{R}^{~i}_{i~\alpha\beta}\wedge\omega^\alpha\wedge\omega^\beta+\tilde{R}^{~i}_{i~\alpha\beta}\omega^\gamma\wedge\omega_\gamma^\alpha\wedge\omega^\beta-\tilde{R}^{~i}_{i~\alpha\beta}\omega^\alpha\wedge \omega^\gamma\wedge\omega_\gamma^\beta)\\
		&+(\tilde{R}^{~i}_{i~\alpha\beta}\omega^n\wedge\omega_n^\alpha\wedge\omega^\beta-\tilde{R}^{~i}_{i~\alpha\beta}\omega^\alpha\wedge \omega^n\wedge\omega_n^\beta)\\
		&+\frac{1}{n-1}[(c-\mathsf{e})_{|\beta}\omega^\beta+(c-\mathsf{e})_{|n}\omega^n]\wedge \sum_{\alpha}(\omega^\alpha\wedge\omega^{\bar{\alpha}})\\
				=&\frac{1}{2}(\tilde{R}^{~i}_{j~\alpha\beta|i}\omega^j+\tilde{R}^{~i}_{i~\alpha\beta,\gamma}\omega^{\bar{\gamma}})\wedge\omega^\alpha\wedge\omega^\beta
		-2\tilde{R}^{~i}_{i~\alpha\beta}\omega^n\wedge\omega^\alpha\wedge \omega^{\bar{\beta}}\\
		&+\frac{1}{n-1}[(c-\mathsf{e})_{|\beta}\omega^\beta+(c-\mathsf{e})_{|n}\omega^n]\wedge \sum_{\alpha}(\omega^\alpha\wedge\omega^{\bar{\alpha}}),\\
	\end{split}
\end{equation*}
therefore
\begin{equation}\label{cor34}
	2\tilde{R}^{~i}_{i~\alpha\beta}\omega^n\wedge\omega^\alpha\wedge \omega^{\bar{\beta}}=\frac{1}{n-1}(c-\mathsf{e})_{|n}\omega^n\wedge \sum_{\alpha}(\omega^\alpha\wedge\omega^{\bar{\alpha}}).
\end{equation}
Since $\tilde{R}^{~i}_{i~\alpha\beta}=-\tilde{R}^{~i}_{i~\beta\alpha}$, both sides of (\ref{cor34}) vanish.
By this fact, (\ref{cor33}) is simplified as
	\begin{equation}
		\pi^*d\xi=-\frac{1}{n-1}d(c-\mathsf{e})\wedge\omega^n.
	\end{equation}
Hence
	\begin{equation}
		0=d(c-\mathsf{e})\wedge(\omega^1\wedge\omega^{n+1}+\cdots+\omega^{n-1}\wedge\omega^{2n-1}),
	\end{equation}
which is equivalent to the following system
	\begin{align*}
		0&=d(c-\mathsf{e})\wedge\omega^1\wedge\omega^{n+1},\\
		&\cdots\cdots\\
		0&=d(c-\mathsf{e})\wedge\omega^{n-1}\wedge\omega^{2n-1}.
	\end{align*}
Therefore $d(c-\mathsf{e})=0$ and thus $d\xi=0$.
	
\end{proof}


\begin{thebibliography}{99}

%\bibitem[Aik]{Aikou10} Tadashi Aikou, \textsl{Some remarks on Berwald manifolds and Landsberg manifolds.}
%Acta Math. Academiae Paedagogicae Ny\'{\i}regyh\'{a}ziensis, Vol. 26, 2010: 139-148.

%\bibitem[Paiva]{Paiva}Juan Carlos \'{A}lvarez Paiva, \textsl{Some problems on Finsler geometry.} Handbook of Differential Geometry, Vol. 2, %North-Holland, Elsevier, 2006: 1-33.

%\bibitem[AntIM]{AntIM} P. L. Antonelli, R. S. Ingarden, M. Matsumoto, \textsl{The Theory of Sprays and Finsler Spaces with Applications in Physics and Biology.} Kluwer Academic Publishers, Netherlands, 1993.


%\bibitem[Asa]{Asanov} G. S. Asanov, \textsl{Finsler Geometry, Relativity and Gauge Theories.}  D. Reidel Publishing Company, Dordrecht, Holland, 1985.


%\bibitem[Asa1]{Asanov06a} G. S. Asanov, \textsl{Finsleroid-Finsler spaces of positive-definite and relativistic types.} Reports on Mathematical Physics, Vol. 58, Issue 2, 2006:  275-300.


%\bibitem[Asa2]{Asanov06b} G. S. Asanov, \textsl{Finsleroid-Finsler space and geodesic spray coefficients.} Publ. Math. Debrecen, Vol. 71, no. 3-4, 2007: 397-412.


%\bibitem[Bao]{Bao} David Bao, \textsl{On two curvature-driven problems in Riemann-Finsler geometry.}
%Advanced Studies in Pure Mathematics, Math. Soc. Japan, Vol.48, 2007: 19-71.

\bibitem{BaoChernShen} David Bao, Shiing-Shen Chern and Zhongmin Shen,
\textsl{An Introduction to Riemann-Finsler Geometry.} Graduate Texts
in Mathematics, Vol. 200, Springer-Verlag, New York, Inc., 2000.

%\bibitem[BaoSh]{BaoShen} David Bao and Zhongmin Shen, \textsl{On the volume of unit tangent spheres in a Finsler manifold.} Results in Mathematics, Vol. 26, 1994: 1-17.
\bibitem{Blai} David E. Blair, \textsl{Riemannian Geometry of Contact and Symplectic Manifolds.} (2nd Edition). Birkh\"{a}user, New York, 2010.

%\bibitem[Bla]{Bla} Wilhelm Blaschke, \textsl{Vorlesungen \"{u}ber Differentialgeometrie II, Affine
%Differentialgeometrie.} Springer, Berlin, 1923.


%\bibitem{Bryant} R. L. Bryant, \textsl{Projectively flat Finsler 2-spheres of constant curvature}. Selecta Math. (New Series), 3(1997), 161-203.

%\bibitem[Bry]{Bryant} Robert L. Bryant, \textsl{Some remarks on Finsler manifolds with constant flag curvature.} Houston J. Math, Vol. 28, No. 2, 2002: 221-262.


%\bibitem[Br1]{Br1} F. Brickell, \textsl{A new proof of Deicke's theorem on homogeneous functions.} Proc. A. M. S., Vol. 16, 1965: 190-191.

%\bibitem[Br2]{Br2} F. Brickell, \textsl{A theorem on homogeneous functions.} J. London Math. Soc., Vol. 42, 1967: 325-329.

\bibitem{ChMS} Xinyue Chen(g), Xiaohuan Mo and Zhongmin Shen, \textsl{On the flag curvature of Finsler metrics of scalar curvature.} J. of London Math. Soc., Vol. 68, No. 2, 2003: 762-780.

\bibitem{ChengShen} Xinyue Cheng and Zhongmin Shen, \textsl{Finsler Geometry-an approach via Randers spaces.} Science Press Beijing and Springer-Verlag Berlin Heidelberg, 2012.

%\bibitem[Chern]{Chern} Shiing-Shen Chern, \textsl{Local equivalence and Euclidean connections in Finsler spaces}. in Chern Selected papers, II, 1989: 95-121.

\bibitem{ChernShen} Shiing-Shen Chern and Zhongmin Shen, \textsl{Riemann-Finsler Geometry.} Nankai Tracts in Mathematics, Vol. 6, World Scientific, 2005.

%\bibitem{ChenHePan} Guangzu Chen, Qun He and Shengliang Pan, \textsl{On weak Berwald $(\alpha,\beta)$-metrics of scalar flag curvature}. Journal of Geometry and Physics, Vol. 86, 2014: 112-121.
%\bibitem[ChoWang]{ChWa} Kai-Seng Chou and Xu-Jia Wang, \textsl{The $L_p$-Minkowski problem and the Minkowski problem in centroaffine geometry.} Advance in Mathematics, Vol. 205, 2006: 33-83.

%\bibitem[Cra]{Cra} M. Crampin, \textsl{On Landsberg spaces and the Landsberg-Berwald problem.} Houston J. Math, Vol. 37, No. 4, 2011: 1103-1124.

%\bibitem{Cra2} Mike Crampin, \textsl{A condition for a Landsberg space to be Berwaldian.} Publ. Math. Debrecen, Vol. 93, no. 1-2, 2018: 143-155.

%\bibitem[Dei]{Deicke} Arno Deicke, \textsl{\"{U}ber die Finsler-Raume mit $A_i=0$.} Arch. Math., Vol. 4, 1953: 45-51.

%\bibitem[Dod]{Dod}C.T.J. Dodson, \textsl{A short review on Landsberg spaces.} MIMS EPrint: 2006.82, 2006: 1-13.

\bibitem{FL} Huitao Feng and Ming Li, \textsl{Adiabatic limit and connections in Finsler geometry.}
Communications in Analysis and Geometry, Vol. 21, No. 3, 2013: 607-624.

%\bibitem[Hil]{H} Roland Hildebrand, \textsl{Centro-affine hypersurface immersions with parallel cubic form.} Beitr\"{a}ge zur Algebra und Geometrie, Volume 56, Issue 2, 2015: 593¨C640.

%\bibitem[HLX]{HLX}Yong Huang, Jiakun Liu, and Lu Xu \textsl{On the uniqueness of $L_p$-Minkowski problems: the constant $p$-curvature case in $\mathbb{R}^3$.} arXiv: 1503.02358v1[math.AP], 9, Mar., 2015.



%\bibitem[Ich1]{Ichijyo76}Yoshihiro Ichijy\={o}, \textsl{Finsler manifolds modeled on a Minkowski space.} J. Math. Kyoto Univ., Vol. 16, 1976: 639-652.

%\bibitem[Ich]{Ichijyo78}Yoshihiro Ichijy\={o}, \textsl{On special Finsler connections with the vanishing hv-curvature tensor.} Tensor(N.S.), Vol. 32, 1978: 149-155.



%\bibitem[JiSh]{JiShen} Min Ji and Zhongmin Shen, \textsl{On Strongly convex indicatrices in Minkowski Geometry.}
%Canad. Math. Bull., Vol. 45, No. 2, 2002: 232-246.

%\bibitem[Kra]{Kr}Peter, Krauter, \textsl{Affine minimal hypersurfaces of rotation.}  Geometriae Dedicata, Vol. 51, 1994: 287-303.

%\bibitem[Lau1]{Laugwitz1} Detlef Laugwitz, \textsl{Zur Differentialgeometrie der Hyperfl\"{a}chen in Vektorr\"{a}umen und zur affingeometrischen Deutung der Theorie der Finsler-R\"{a}ume.}
%Math. Z., Vol. 67, 1957: 63-74.

%\bibitem[Lau2]{Laugwitz2} Detlef Laugwitz, \textsl{Eine Beziehung zwischen affiner und Minkowskischer Differentialgeometrie.}
%Publ. Math. Debrecen, Vol. 5, 1957: 72-76.

%\bibitem{LLSSW}Anmin Li, Huili Liu, A. Schwenk-Shellschmidt, Udo Simon and Changping Wang, \textsl{Cubic form methods and relative Tchebychev Hypersurfaces.}
%Geometriae Dedicata, Vol. 66, 1997: 203-221.

\bibitem{LSZ}Anmin Li, Udo Simon and Guosong Zhao, \textsl{Global Affine Differential Geometry of Hypersurfaces.} W. de Gruyter, Berlin-New York, 1993.

%\bibitem{LiW} Anmin Li and Changping Wang, \textsl{Canonical centroaffine hypersurfaces in $\mathbb{R}^{n+1}$.} Results in Math., Vol. 20, 1991: 660-681.



%\bibitem{L1} Xingxiao Li, \textsl{On the correspondence between symmetric equiaffine hyperspheres and minimal symmetric Lagrange submanifolds (in Chinese).}
%Sci. Sin. Math., Vol. 44, No. 1, 2014: 13-36.

%\bibitem{L2} Xingxiao Li, \textsl{A classification theorem of nondegenerate equiaffine symmetric hypersurfaces.}  ArXiv:1408.5947v1, 2014.

%\bibitem{LZ} Xingxiao Li and Guosong Zhao, \textsl{On the equiaffine symmetric hyperspheres.}  ArXiv:1408.4317v1, 2014.


%\bibitem{LSW} Huili Liu, Udo Simon and Changping Wang, \textsl{Conformal Structure in affine geometry: complete Tchebychev hypersurface.}
%Abh. Math. Sem. Univ. Hamburg 66 (1996), 249-262

%\bibitem{LW1} Huili Liu and Changping Wang, \textsl{The centroaffine Tchebychev operator.} Results in Math. Vol.27, 1995: 77-92.


%\bibitem{LW} Huili Liu and Changping Wang, \textsl{Relative Tchebychev surfaces in $\mathbb{R}^3$.} Kyushu J. Math. Vol.50, 1996: 553-540.

%\bibitem{LW97} Huili Liu and Changping Wang, \textsl{Centroaffine surfaces with parallel traceless cubic form.} Bull. Belg. Math. Soc.,  Vol.4, 1997: 493-499.
\bibitem{Li} Ming Li, \textsl{Equivalence theorems of Minkowski spaces and applications in Finsler geometry.}(in Chinese) Acta Math. Sinica (Chin. Ser.) Vol. 62, No.2, 2019: 177-190. (see arXiv:1504.04475v2  for the English version.)





%\bibitem[Li2]{Li2} Ming Li, \textsl{Rigidity theorems for relative Tchebychev hypersurfaces.} Results in Math., 2015.

\bibitem{Lizhang} Ming Li and Lihong Zhang, \textsl{Properties of Berwald scalar curvature.} Front. Math. China, 15(6), 2020: 1143-1153.

%\bibitem[LuWang]{LuWang}J. Lu and X.-J. Wang, \textsl{Rotationally symmetric solutions to the $L_p$-Minkowski problem,} J. Differential Equations, 254, (2013): 983-1005.

%\bibitem[Lut]{Lut} Erwin Lutwak, \textsl{The Brunn-Minkowski-Firey Theory II: affine and Geominimal surface Areas}. Advances in Mathematics, Vol. 118, 1996: 244-249.

%\bibitem[LYZ]{LYZ}E. Lutwak, D. Yang and G. Zhang, \textsl{On the Lp-Minkowski problem.} Trans. Amer. Math. Soc. 356 (2004), 4359-4370.

%\bibitem[Mat1]{Matsumoto72} Makoto Matsumoto, \textsl{On C-reducible Finsler spaces,} Tensor N. S., 24, 1972: 29-37.


%\bibitem[Mat2]{Matsumoto} Makoto Matsumoto, \textsl{Theory of Finsler spaces with ($\alpha,\beta$)-metric.}
%Reports on Mathematical Physics, Vol. 31, no. 1, 1992: 43-83.

%\bibitem[MatS]{MatShi} Makoto Matsumoto  and Chi\={o}k\={o} Shibata, \textsl{On semi-C-reducibility, T-tensor=0 and S4-likeness of Finsler spaces.}
%J. Math. Kyoto. Univ., Vol. 19, no. 2, 1979: 301-314.


%\bibitem[Matv]{Matveev09} Vladimir S. Matveev, \textsl{On `All regular Landsberg metrics are Berwald'' by Z. I. Szab\'{o}.}
%Balkan Journal of Geometry and Its Applications, Vol. 14, No. 2, 2009: 50-52.

\bibitem{Mo} Xiaohuan Mo, \textsl{An Introduction to Finsler Geometry.} Peking University series in Math., Vol. 1, World Scientific Publishing Co. Pte. Ltd., 2006.


%\bibitem[MoHu]{MoHuang} Xiaohuan Mo and Libing Huang, \textsl{On characterizations of Randers norms in a Minkowski space.}
%International Journal of Mathematics, Vol. 21, No. 4, 2010: 523-535.

%\bibitem[Muz]{Muz}Z. Muzsnay, \textsl{The Euler-Lagrange PDE and Finsler metrizability.} Houston J. Math., Vol. 32, No. 1, 2006:
%79-98.

\bibitem{NajafiTa} Behzad Najafi and Akbar Tayebi, \textsl{Weakly stretch Finsler metrics.} Publ. Math. Debrecen, Vol. 91, 2017: 1-12.

\bibitem{NS} Katsumi Nomizu and Takeshi Sasaki, \textsl{Affine Differential Geometry.} Cambridge University Press, 1994.


%\bibitem[Run]{Rund} Hanno Rund, \textsl{The differential Geometry of Finsler Spaces.} Springer-Verlag, 1959.

\bibitem{Schneider1} Rolf Schneider, \textsl{Zur affinen Differentialgeometrie im Gro{\ss}en. I.} Math. Zeitschr., Vol. 101, 1967: 375-406.

%\bibitem{Schneider2} Rolf Schneider, \textsl{Zur affinen Differentialgeometrie im Gro{\ss}en. II.} Math. Zeitschr., Vol. 102, 1967: 1-8.

%\bibitem[Sch3]{Schneider2.1} Rolf Schneider, \textsl{\"{U}ber die Finslerr\"{a}ume mit $S_{ijkl}=0$.} Math. Zeitschr., Vol. 102, 1967: 1-8.


%\bibitem[Sch4]{Schneider3} Rolf Schneider, \textsl{Convex Bodies: The Brunn-Minkowski Theory(Second Expanded Edition).} Canbridge University Press, Cambridge, 2014.


\bibitem{SS} Yibing Shen and  Zhongmin Shen, \textsl{Introduction to Modern Finsler Geometry.} World Scientific, Singapore, 2016.

\bibitem{shen97} Zhongmin Shen, \textsl{Volume comparison and its applications in Riemann-Finsler geometry.} Adv. in Math., Vol. 128, 1997: 306-328.

\bibitem{Shenbook1}Zhongmin Shen, \textsl{Differential Geometry of Spray and Finsler
Spaces.} Kluwer Acad. Publ., 2001.

\bibitem{Shenbook2}Zhongmin Shen, \textsl{Lectures on Finsler Geometry.} World Scientific, 2001.


%\bibitem{Shen} Zhongmin Shen, \textsl{Landsberg curvature, S-curvature and Riemann curvature.} Riemann-Finsler Goemetry, MSRI Publications. Vol. 50, 2004: 303-355.

%\bibitem[She]{Shen09}Zhongmin Shen, \textsl{On a class of Landsberg metrics in Finsler geometry.} Canad. J. Math., Vol. 61, No. 6, 2009: 1357-1374.

%\bibitem[Shi]{Shima} Hirohiko Shima, \textsl{The Geometry of Hessian Structure.} World Scientific Publishing Co. Pte. Ltd., Singapore, 2007.


%\bibitem[SziLK]{SziLK} J\'{o}zsef Szilasi, Rezs\H{o} L Lovas and D\'{a}vid Cs Kert\'{e}sz, \textsl{Connections, Sprays and Finsler Structures.}World Scientific Publishing Co. Pte. Ltd., Singapore, 2014.


\bibitem{SSV} Udo Simon, Angela Schwenk-Schellschmidt and Helmut Viesel, \textsl{Introduction to the Affine Differential Geometry of Hypersurfaces.}
Lecture notes, Science University Tokyo, 1991. %ISBN 3-7983-1529-9,

%\bibitem[Sim]{Simon} Udo Simon, \textsl{The centroaffine volume of generalized geodesic balls under inversion at the sphere.} Results in Math., Vol. 43, 2003: 343-358.

%\bibitem[Su]{Su} Buchin Su, \textsl{On the Theory of Surfaces in the Affine Space
%VII. Affine Moulding Hypersurfaces and Affine Hypersurfaces of Revolution.}
%Japanese Journal of Mathematics, Vol. 6, 1929: 29-42.

%\bibitem[S\"{u}s]{Sue} Wilhelm S\"{u}ss, \textsl{Affinrotationshyperfl\"{a}chen.}
%Japanese Journal of Mathematics, Vol. 5, 1928: 85-95.

%\bibitem{STV} Udo Simon, Houda Trabelsi and Luc Vrancken, \textsl{Complete hyperbolic Tchebychev hypersurfaces.} J. Geom., Vol. 89, 2008: 148-159.



%\bibitem[Sza1]{Szabo81} Zolt\'{a}n Imre Szab\'{o}, \textsl{Positive definite Berwald spaces (structure theorems on Berwald spaces).} Tensor, N. S., Vol. 35, 1981: 25-39.


%\bibitem[Sza1]{Szabo08} Zolt\'{a}n Imre Szab\'{o}, \textsl{All regular Landsberg metrics are Berwald.} Ann. Glob. Anal. Geom., Vol. 34, 2008: 381-386.

%\bibitem[Sza2]{Szabo09} Zolt\'{a}n Imre Szab\'{o}, \textsl{Correction to ``All regular Landsberg metrics are Berwald''.} Ann. Glob. Anal. Geom., Vol. 35, 2009: 227-230.

%\bibitem[Tor]{Tor}Ricardo Gallego Torrom\'{e}, \textsl{ON THE BERWALD-LANDSBERG PROBLEM.} arXiv:1110.5680v4, [math.DG], 7. May, 2014.

%\bibitem[Var]{Varga} Ott\'{o} Varga, \textsl{Die Kr\"{u}mmung der Eichfl\"{a}che des Minkowskischen Raumes und die geometrische Deutung des einen Kr\"{u}mmungstensors des Finslerschen Raumes.}
%Abh. Math. Sem. Univ. Hamburg 20, 1955: 41-51.


%\bibitem[Wang]{Wang1} Changping Wang, \textsl{Centroaffine minimal hypersurfaces in $\mathbb{R}^{n+1}$.} Geometriae Dedicata, Vol. 51, 1994: 63-74.

%\bibitem[YanYL]{YYL}Yun Yang, Yanhua Yu and Huili Liu, \textsl{Centroaffine geometry of equiaffine rotation surfaces in $\mathbb{R}^3$.}
%Journal of Mathematical Analysis and Applications, Vol. 414, 2014: 46-60.

%\bibitem[YuZ]{YuZh} Changtao Yu and hongmei Zhu, \textsl{On a new class of Finsler metrics.} Differential Geometry and its Applications, Vol. 29, 2011: 244-254.

%\bibitem[ZhangG]{Zhangg} Gaoyong Zhang, \textsl{New affine isoperimetric inequalities.} ICCM 2007, Vol, II: 239-267.

\bibitem{Zhang} Weiping Zhang, \textsl{Lectures on Chern-Weil Theory and Witten Deformations}. Nankai Tracts in Mathematics, Vol. 4, World Scientific Publishing Co. Pte. Ltd., 2001.

%\bibitem[ZhoL]{ZL}Shasha Zhou and Benling Li, \textsl{On Landsberg general $(\alpha,\beta)$-metrics with a conformal 1-form.} Differential Geometry and its Applications, Vol. 59, 2018: 46-65¡£

%\bibitem[ZhoWL]{ZWL}Shasha Zhou, Jiayue Wang and Benling Li, \textsl{On a class of almost regular Landsberg metrics.} SCIENCE CHINA Mathematics. Vol. 62, 2019: 935-960.

%\bibitem[ZohM]{ZohM} M. Zohrehvand and H. Maleki, \textsl{On general $(\alpha,\beta)$-metrics of Landsberg type.} International Journal of Geometric Methods in Modern Physics, Vol. 13, No. 6, 2016: 1650085.

%\bibitem{ZouCheng14}Yangyang Zou and Xinyue Cheng, \textsl{The generalized unicorn problem on $(\alpha,\beta)$-metrics.} J. Math. Anal. Appl., Vol. 414, 2014: 574-589.

\end{thebibliography}
\end{document}